\documentclass[12pt,leqno]{article}
\usepackage{amssymb, amscd, amsmath, amsthm}
\usepackage{dsfont}
\allowdisplaybreaks

\def\beq{\begin{equation}}
\def\eeq{\end{equation}}

\theoremstyle{definition}

\newtheorem{problem}{Problem}

\newtheorem*{rema}{Remark}

\theoremstyle{plain}
\newtheorem{theorem}{Theorem}
\newtheorem{claim}{Claim}
\newtheorem{fact}{Fact}
\newtheorem{lemma}{Lemma}

\newtheorem{corollary}{Corollary}
\newtheorem{proposition}{Proposition}

\numberwithin{equation}{section}
\numberwithin{proposition}{section}
\numberwithin{observation}{section}
\numberwithin{definition}{section}
\numberwithin{theorem}{section}
\numberwithin{problem}{section}
\numberwithin{example}{section}
\numberwithin{claim}{section}
\numberwithin{fact}{section}
\numberwithin{lemma}{section}
\numberwithin{conjecture}{section}
\numberwithin{corollary}{section}

\begin{document}

\title{On the maximum of the sum of the sizes of non-trivial cross-intersecting families}

\author{by P. Frankl
\\
R\'enyi Institute, Budapest, Hungary\
}

\date{}
\maketitle

\begin{abstract}
Let $n \geq 2k \geq 4$ be integers, ${[n]\choose k}$ the collection of $k$-subsets of $[n] = \{1, \ldots, n\}$.
Two families $\mathcal F, \mathcal G \subset {[n]\choose k}$ are said to be \emph{cross-intersecting} if $F \cap G \neq \emptyset$ for all $F \in \mathcal F$ and $G \in \mathcal G$.
A family is called non-trivial if the intersection of all its members is empty.
The best possible bound $|\mathcal F| + |\mathcal G| \leq {n \choose k} - 2 {n - k\choose k} + {n - 2k \choose k} + 2$ is established under the assumption that $\mathcal F$ and $\mathcal G$ are non-trivial and cross-intersecting.
For the proof a strengthened version of the so-called \emph{shifting technique} is introduced.
\end{abstract}

\section{Introduction}
\label{sec:1}

Let $n, k$ be integers, $n \geq 2k \geq 4$.
Let $[n] = \{1, \ldots, n\}$ be the standard $n$-element set, ${[n]\choose k}$ the collection of its $k$-subsets.
Subsets of ${[n]\choose k}$ are called \emph{$k$-graphs} or \emph{$k$-uniform} families.
A family $\mathcal F \subset {[n]\choose k}$ is called \emph{intersecting} if $F \cap F' \neq \emptyset$ for all $F, F'\in \mathcal F$.
Similarly, if $F \cap G \neq \emptyset$ for all $F \in \mathcal F$ and $G \in \mathcal G$ then $\mathcal F, \mathcal G \subset {[n]\choose k}$ are called \emph{cross-intersecting}.

For a family $\mathcal F$ set $\cap \mathcal F = \cap \{F : F \in \mathcal F\}$.
If $\cap \mathcal F = \emptyset$ then $\mathcal F$ is called \emph{non-trivial} and if $\cap \mathcal F \neq \emptyset$ then it is called a \emph{star}.
For $i \in [n]$ let $\mathcal S_i = \left\{S \in {[n]\choose k} : i \in S\right\}$ be the \emph{full star}, $k$ is understood from the context.
Note that $|\mathcal S_i| = {n - 1\choose k - 1}$.

Let us recall the Erd\H{o}s--Ko--Rado Theorem, one of the central results in extremal set theory.

\begin{theorem} [{\cite{EKR}}]
\label{th:1.1}
Suppose that $\mathcal F \subset {[n] \choose k}$, $n \geq 2k \geq 4$ and $\mathcal F$ is intersecting.
Then
\beq
\label{eq:1.1}
|\mathcal F| \leq {n - 1\choose k - 1}.
\eeq
\end{theorem}

The full star shows that \eqref{eq:1.1} is best possible.

The Hilton--Milner Theorem shows in a strong way that for $n > 2k$ only full stars achieve equality in \eqref{eq:1.1}.

\begin{theorem}[\cite{HM}]
\label{th:1.2}
Suppose that $\mathcal F \subset  {[n]\choose k}$, $n > 2k \geq 4$, $\mathcal F$ is intersecting and non-trivial.
Then
\beq
\label{eq:1.2}
|\mathcal F| \leq {n - 1\choose k - 1} - {n - k - 1\choose k - 1} + 1.
\eeq
\end{theorem}

For $k \neq 3$ the only family providing equality in \eqref{eq:1.2} is the Hilton--Milner Family,
$\mathcal H(n, k) = \left\{H \in {[n]\choose k} : 1 \in H, [2, k + 1] \cap H \neq \emptyset \right\} \cup \{[2, k + 1]\}$.
For $k = 3$ the \emph{triangle family} $\mathcal T(n, k) = \left\{T \in {[n]\choose k} : |T \cap [3]| \geq 2\right\}$ is the only other family attaining the bound \eqref{eq:1.2}.

By now there are dozens of papers proving and reproving results related to these basic theorems (\cite{A}, \cite{AK}, \cite{B}, \cite{D}, \cite{F2}, \cite{F4}, \cite{FF}, \cite{FK}, \cite{FT}, \cite{HK}, \cite{KZ}, \cite{M}, \cite{P}, etc.).

The author might lack modesty, but he pretends to have found a closely related natural question that has not been investigated before.

\setcounter{problem}{2}
\begin{problem}
\label{pr:1.4}
Let $\mathcal F, \mathcal G \subset {[n]\choose k}$ be non-trivial cross-intersecting families.
Determine or estimate
\beq
\label{eq:1.4}
h(n, k) := \max \bigl\{|\mathcal F| + |\mathcal G|\bigr\}.
\eeq
\end{problem}

The construction that we propose is really simple.
Let $\mathcal A = \bigl\{A_1, A_2\bigr\}$ where $A_1, A_2 \in {[n]\choose k}$ are disjoint.
Set
$$
\mathcal B = \left\{ B \in {[n]\choose k} : B \cap A_i \neq \emptyset, \ i = 2\right\}.
$$

Clearly,
$|\mathcal B| = {n\choose k} - 2{n - k\choose k} + {n - 2k\choose k}$.
Note that for $k$ fixed and $n \to \infty$,
$$
|\mathcal B| = k^2 {n - 2\choose k - 2} + O\left({n - 3\choose k - 3}\right).
$$

\setcounter{theorem}{3}

\begin{theorem}
\label{th:1.5}
Let $n > 2k \geq 4$ be integers.
Then
\beq
\label{eq:1.5}
h(n, k) = 2 + {n\choose k} - 2{n - k\choose k} + {n - 2k\choose k},
\eeq
moreover for $k \geq 3$ up to automorphism the above example is unique.
\end{theorem}

Let us recall the \emph{shifting partial order} that can be traced back to Erd\H{o}s, Ko and Rado \cite{EKR}.
For two $k$-sets $A = \{x_1, \ldots, x_k\}$ and $B = \{y_1, \ldots, y_k\}$ where $x_1 < x_2 < \ldots < x_k$,
$y_1 < \ldots < y_k$ we say that $A$ \emph{precedes} $B$ and denote it by $A \prec B$ if $x_i \leq y_i$ for all
$1 \leq i \leq k$.
A family $\mathcal F \subset {[n]\choose k}$ is called \emph{initial} or \emph{shifted} if $A \prec B$ and $B \in \mathcal F$ always imply $A \in \mathcal F$.

Note that the full star $\mathcal S_1$ and the Hilton--Milner Family as well as $\mathcal T(n, k)$ and many other important families are initial.

Erd\H{o}s, Ko and Rado invented the shifting operator (cf.\ definition below) that maintains the size of a family along the intersection or cross-intersection properties.
It might destroy non-triviality, but there are certain ways to circumvent this difficulty (cf.\ \cite{FF}).

In great contrast the families $\mathcal A$ and $\mathcal B$ defined above are \emph{not} initial.
In fact the answer for initial families is completely different.
Define $\mathcal P = {[k + 1]\choose k}$, $\mathcal R = \left\{R \in {[n]\choose k} : |R \cap [k + 1]| \geq 2\right\}$.
It is easy to check that $\mathcal P$ and $\mathcal R$ are non-trivial and cross-intersecting.

\begin{theorem}
\label{th:1.6}
Suppose that $\mathcal F, \mathcal G \subset {[n]\choose k}$ are non-trivial, cross-intersecting initial families, then for $n \geq 2k \geq 4$,
\beq
\label{eq:1.6}
|\mathcal F| + |\mathcal G| \leq k + 1 + \sum_{2 \leq i \leq k} {k + 1\choose i} {n - k - 1\choose k - i}.
\eeq
\end{theorem}

Note that for fixed $k$ and $n \to \infty$ the RHS is ${k + 1\choose 2} {n - 2\choose k - 2} + O\left({n - 3\choose k - 3}\right)$, i.e., it is asymptotic to $\frac{k + 1}{2k} h(n, k)$.

For the proof of Theorem \ref{th:1.5} we need the following old result.

\begin{theorem}[\cite{FT}]
\label{th:1.7}
Let $k \geq \ell > 0$ be integers.
Suppose that $n \geq k + \ell$, $\mathcal F \subset {[n]\choose k}$, $\mathcal G \subset {[n]\choose \ell}$ and the families $\mathcal F, \mathcal G$ are non-empty and cross-intersecting.
Then
\beq
\label{eq:1.7}
|\mathcal F| + |\mathcal G| \leq {n\choose k} - {n - \ell \choose k} + 1.
\eeq
Moreover unless $n = k + \ell$ or $k = \ell = 2$ the equality is strict for $|\mathcal G| > 1$.
\end{theorem}

Since in \cite{FT} uniqueness was not proven, we provide a full proof of \eqref{eq:1.7}.
This proof uses no computation what the reader might find nice.

Let us recall some standard notations.
For subsets $A,\, B$ define $\mathcal F(A) = \{ F \setminus A : A \subset F \in \mathcal F\}$,
$\mathcal F(\overline{B}) = \left\{F \in \mathcal F: F \cap \overline B = \emptyset \right\}$ and if
$A \cap B = \emptyset$ then set $\mathcal F(A, \overline B) = \left\{F \setminus A : A \subset F \in \mathcal F, F \cap \overline B = \emptyset\right\}$.
If $A = \{i\}$, $B = \{j\}$ then we use the shorthand notations $\mathcal F(i)$, $\mathcal F(\overline j)$, $\mathcal F(i, \overline j)$, etc.
Let us mention that $\mathcal F(A, \overline B)$ is a family on the ground set $[n]\setminus (A \cup B)$.
The notation $\mathcal F(A, A \cup B)$ is also quite common for disjoint sets $A, B$.
Note that $\mathcal F(A, A \cup B) = \mathcal F(A, \overline B)$ in this case.

\section{Shifting and some more tools}
\label{sec:2}

For a family $\mathcal F \subset {[n]\choose k}$ and integers $1 \leq i \neq j \leq n$ one defines the shifting (operator) $S_{i   j}$ by $$
S_{i   j}(\mathcal F) = \left\{S_{i   j}(F) : F \in \mathcal F \right\}
$$
where
$$
S_{i   j}(F) = \begin{cases}
F' := (F \setminus \{j\}) \cup \{i\}\ &\text{ if }\ i \notin F,\ j \in F \ \text{ and } \ F' \notin \mathcal F,\\
F &\text{ otherwise.}
\end{cases}
$$

Note that $\left|S_{i   j} (\mathcal F)\right| = |\mathcal F|$ and $S_{i   j}(\mathcal F) \subset {[n]\choose k}$.
It is well known (cf.\ \cite{F3}) that if $\mathcal F, \mathcal G$ are cross-intersecting, then $S_{i  j}(\mathcal F)$ and $S_{i   j}(\mathcal G)$ are also.

In the present paper we are mostly dealing with non-trivial families.
However it might happen that $\mathcal F$ is non-trivial but $S_{i  j}(\mathcal F)$ is a star.
The next statement is easy to prove.

\begin{fact}[cf. e.g.\ \cite{FW}]
\label{fact:2.1}
Suppose that $\mathcal F \subset {[n]\choose k}$ is non-trivial but $S_{i   j}(\mathcal F)$ is a star.
Then $S_{i   j}(\mathcal F) \subset \mathcal S_i$ and

\phantom{i}{\rm (i)} \ $\mathcal F(i) \cap \mathcal F(j) = \emptyset$,

{\rm (ii)} $F \cap \{i, j\} \neq \emptyset$ for all $F \in \mathcal F$.
\end{fact}

The cross-intersecting families $\mathcal F, \mathcal G \subset {[n]\choose k}$ are called \emph{saturated} if they cease to be cross-intersecting upon the addition of any new $k$-set $H \in {[n]\choose k}$.

For a family $\mathcal F \subset {[n]\choose k}$ and a positive integer $\ell$ define
$$
\mathcal T^{(\ell)}(\mathcal F) = \left\{H \in {[n]\choose \ell} : H \cap F \neq \emptyset\ \forall F \in \mathcal F\right\}.
$$
With this notation $\mathcal F$ and $\mathcal G$ are saturated iff $\mathcal F = \mathcal T^{(k)}(\mathcal G)$ and $\mathcal G = \mathcal T^{(k)}(\mathcal F)$.

\setcounter{claim}{1}
\begin{claim}
\label{cl:2.2}
Suppose that $\mathcal F, \mathcal G \subset {[n]\choose k}$ are cross-intersecting and saturated and $F \cap \{i, j\} \neq \emptyset$ for all $F \in \mathcal F$ (equivalently $\{i, j\} \in \mathcal T^{(2)}(\mathcal F)$).
Then $\mathcal S_{ij} = \left\{ S \in {[n]\choose k} :\{i, j\} \subset S\right\} \subset \mathcal G$. \hfill $\square$
\end{claim}

\begin{claim}
\label{cl:2.3}
Let $1 \leq i \neq j \leq n$, $\mathcal F \subset {[n]\choose k}$.
Then
$$
\left|\mathcal T^{(2)}(\mathcal F)\right| \leq \left|\mathcal T^{(2)}\left(S_{i   j}(\mathcal F)\right)\right|.
$$
\hfill $\square$
\end{claim}

For a family $\mathcal F \subset {[n]\choose k}$ and an integer $\ell$, $0 \leq \ell \leq k$ define the \emph{$\ell$-shadow} $\sigma_\ell(\mathcal F) := \left\{H \in {[n]\choose \ell}: \exists F \in \mathcal F, H \subset F\right\}$.

Let us recall the shadow theorem of Sperner \cite{S}.
\beq
\label{eq:2.1}
\left|\sigma_\ell(\mathcal F)\right| \Bigr/{n\choose \ell} \geq |\mathcal F|\Bigr/{n\choose k}.
\eeq

This has an important consequence for us.

\setcounter{lemma}{3}
\begin{lemma}
\label{lem:2.4}
Let $m, a, b$ be positive integers, $m \geq a + b$.
Suppose that $\mathcal A \subset {[m]\choose a}$ and $\mathcal B \subset {[m]\choose b}$ are cross-intersecting.
Then
\beq
\label{eq:2.2}
\frac{|\mathcal A|}{{m\choose a}} + \frac{|\mathcal B|}{{m\choose b}} \leq 1.
\eeq
\end{lemma}

\begin{proof}
Define the family of complements $\mathcal B^c = \left\{ [m] \setminus B, B \in \mathcal B\right\}$.
By the cross-intersecting property $\mathcal A \cup \mathcal B^c$ is an antichain.
Consequently $\sigma_a(\mathcal B^c) \cap \mathcal A = \emptyset$.
In view of \eqref{eq:2.1} $\frac{|\mathcal B|}{{n\choose b}} {n \choose a} + |\mathcal A| \leq {n\choose a}$ which is equivalent to \eqref{eq:2.2}.
\end{proof}

Let us prove a corollary of \eqref{eq:2.1} for initial families.

\begin{lemma}
\label{lem:2.5}
Let $\mathcal F \subset {[n]\choose k}$ be initial and $q$ a positive integer, $q < n$.
Let $P \subset R \subset [q]$ with $|R| \leq k$.
Then
\beq
\label{eq:2.3}
\left|\mathcal F\left(P, \overline{[q]\setminus P}\right)\right|\Bigr/{n - q\choose k - |P|} \leq \left|\mathcal F\left(R, \overline{[q]\setminus R}\right)\right|\Bigr/{n - q\choose k - |R|}.
\eeq
\end{lemma}

\begin{proof}
Set $d = |R\setminus P|$.
Note that $R \subset [q]$ implies $R \setminus P \prec S$ for every $S \in {[q + 1, n]\choose |R \setminus P|}$.
Take a set $F \in \mathcal F$ satisfying $F \cap [q] = P$ and a $d$-element subset $S$ of $F \setminus [q]$.
Then $(F\cup (R \setminus P)) \setminus S \prec F$ implies that $F \setminus ([q] \cup S)$ is in $\mathcal F\left(R, \overline{[q]\setminus R}\right)$.
Now \eqref{eq:2.3} follows from \eqref{eq:2.1}.
\end{proof}

\section{The proof of Theorem \ref{th:1.6}}
\label{sec:3}

Let us define $H_i = [k + 1] \setminus \{i\}$ for $1 \leq i \leq k + 1$.
Let $H \subset {[2, n]\choose k}$.
Then $H_{k + 1} \prec H_k \prec \ldots \prec H_1 \prec H$.
Let $\mathcal H = \mathcal F$ or $\mathcal G$.
Non-triviality implies $\mathcal H \cap {[2, n]\choose k} \neq \emptyset$.
By initiality ${[k + 1]\choose k} = \left\{ H_1, \ldots, H_{k + 1}\right\} \subset \mathcal H$ follows.
By cross-intersection $\left|H\cap [k + 1]\right| \geq 2$ for all $H \in \mathcal F \cup \mathcal G$.
For $P \subset [k + 1]$, $2 \leq |P| \leq k$ define $\mathcal F(P) = \mathcal F\left(P, \overline{[k + 1] \setminus P}\right)$ and
$\alpha(P) = |\mathcal F(P)|\Bigr/{n - k - 1\choose k - |P|}$.
Define $\beta(P)$ analogously using $\mathcal G$ instead of $\mathcal F$.

In view of Lemma \ref{lem:2.5} for $P \subset R \subsetneqq [k + 1]$, $\alpha(P) \leq \alpha(R)$.

\begin{claim}
\label{cl:3.1}
If $P \cap Q = \emptyset$ then
\beq
\label{eq:3.1}
\alpha(P) + \beta(Q) \leq 1.
\eeq
\end{claim}

\begin{proof}
If $Q = [k + 1]\setminus P$ then $\mathcal F(P)$ and $\mathcal G(Q)$ are cross-intersecting and $(k - |P|) + k - |Q| = k - 1 \leq n - (k + 1)$ implies that we can apply Lemma \ref{lem:2.4}.
Then \eqref{eq:3.1} follows from \eqref{eq:2.2}.
If $Q \subsetneqq [k + 1] \setminus P$ then we can deduce \eqref{eq:3.1} using monotonicity: $\beta(Q) \leq \beta([k + 1]\setminus P)$.

Let us use \eqref{eq:3.1} to prove the main lemma.

\setcounter{lemma}{1}
\begin{lemma}
\label{lem:3.2}
Let $2 \leq i \leq \frac{k}{2}$ and let $P, Q \in {[k + 1]\choose i}$ be disjoint.
Then
\beq
\label{eq:3.2}
|\mathcal F(P)| + |\mathcal G(Q)| + |\mathcal G([k + 1]\setminus P) + |\mathcal F([k + 1] \setminus Q)| \leq {n - k - 1\choose k - i} + {n - k - 1\choose i - 1}.
\eeq
\end{lemma}

\begin{proof}
From \eqref{eq:3.1} we derive the following three inequalities
\begin{align}
\label{eq:3.3}
|\mathcal F(P)| + |\mathcal G(Q)| &\leq {n - k - 1\choose k - i},\\
\label{eq:3.4}
\alpha(P) + \beta([k + 1]\setminus P) &\leq 1,\\
\label{eq:3.5}
\beta(Q) + \alpha([k + 1]\setminus Q) &\leq 1.
\end{align}
Note that $2i \leq k$ implies $i - 1 < k - i$ and for $n \geq 2k$, $(k - i) + (i - 1) = k - 1 \leq n - k - 1$.
Thus ${n - k - i\choose i - 1} \leq {n - k - 1\choose k - i}$.
Set $\gamma = {n - k - 1\choose i - 1}\Bigr/ {n - k - 1 \choose k - i}$.
Multiplying by ${n - k - 1\choose i - 1}$ the sum of \eqref{eq:3.4} and \eqref{eq:3.5} yields
$$
|\mathcal G([k + 1] \setminus P)| + |\mathcal F((k + 1)\setminus Q)| + \gamma\bigl(|\mathcal F(P)| + |\mathcal G(Q)|\bigr) \leq 2 {n - k - 1\choose i - 1}.
$$
Adding $(1 - \gamma)$ times \eqref{eq:3.3} yields \eqref{eq:3.2}.
\end{proof}

To deduce Theorem \ref{th:1.6} from Lemma \ref{lem:3.2} is easy.
Let us define for $0 \leq i \leq k$,
$f_i = \bigl|\{F \in \mathcal F : |F \cap [k + 1]| = i\}\bigr|$.
Define $g_i$ analogously.
Obviously,
$$
|\mathcal F| = f_0 + \ldots + f_k = f_2 + \ldots + f_{k - 1} + {k + 1\choose k}, \qquad |\mathcal G| = g_2 + \ldots + g_{k - 1} + {k + 1\choose k}.
$$
For a fixed $i$, $2 \leq i \leq \frac{k}2$, averaging \eqref{eq:3.2} over all choices of disjoint $i$-sets $P$ and $Q$ yields
$$
f_i + g_i + f_{k + 1 - i} + g_{k + 1 - i} \leq {k + 1\choose i} {n - k - 1\choose k - i} + {k + 1\choose k + 1 - i} {n - k - 1\choose i - 1}.
$$
In case of $k + 1 = 2q$ is even summing \eqref{eq:3.3} over all ordered complementary pairs $(P, Q)$ we infer
$$
f_q + g_q \leq {k + 1\choose q} {n - k - 1\choose k - q}.
$$
Summing these inequalities we obtain
$$
|\mathcal F| + |\mathcal G| \leq \sum_{2 \leq j \leq k - 1} {k + 1\choose j} {n - k - 1\choose k - j} + 2 {k + 1\choose k} \quad \text{ proving \eqref{eq:1.6}.}
$$
In case of equality, equality must hold all the way, in particular in \eqref{eq:3.3} for all $2 \leq i \leq \frac{k + 1}{2}$ and all pairs of disjoint $i$-sets $P$ and $Q$.
Using uniqueness in case of Sperner's shadow theorem \eqref{eq:3.1}, we infer that one of $\mathcal F(P)$ and $\mathcal G(Q)$ must be empty and the other the full set ${[k + 2, n]\choose k - i}$.
Using equality in Lemma \ref{lem:2.5} we eventually arrive at the conclusion that either $\mathcal F = {[k + 1]\choose k}$ or $\mathcal G = {[k + 1]\choose k}$ and the other is $\left\{H \in {[n]\choose k} : |H \cap [k + 1]| \geq 2\right\}$.
\end{proof}

Let us prove the analogous result for families of distinct uniformities as well.

\setcounter{theorem}{2}
\begin{theorem}
\label{th:3.3}
Suppose that $n \geq k + \ell$, $k \geq \ell \geq 2$.
Let $F \subset {[n]\choose k}$ and $\mathcal G \subset {[n]\choose \ell}$ be non-trivial, shifted and cross-intersecting.
Then
\beq
\label{eq:3.6}
|\mathcal F| + |\mathcal G| \leq \sum_{2 \leq i \leq \ell + 1} {\ell + 1\choose i} {n - \ell - 1\choose k - i} + {\ell + 1\choose \ell} =: g(n, k, \ell).
\eeq
\end{theorem}

Note that for $k = \ell$, ${n - \ell - 1\choose k - \ell - 1} = 0$ and \eqref{eq:3.6} follows from Theorem \ref{th:1.6}.
In the sequel $k > \ell \geq 2$.
By non-triviality and shiftedness $[2, k + 1] \in \mathcal F$ and $[2, \ell + 1] \in \mathcal G$.
Then by shiftedness ${[k + 1]\choose k} \subset \mathcal F$ and ${[\ell + 1]\choose \ell} \subset \mathcal G$.
The cross-intersecting property implies $|H| \geq 2$ for all $H \in \mathcal F \cup \mathcal G$.

For the proof of \eqref{eq:3.6} we need an inequality and an identity.

A sequence $(c_0, c_1, \ldots, c_m)$ of length $m + 1$ is called \emph{appropriate} if $c_i = c_{m - i}$, $0 \leq i < \frac{m}2$ and
$0 \leq c_0 \leq \ldots \leq c_{\lfloor m / 2\rfloor}$.

\setcounter{lemma}{3}
\begin{lemma}
\label{lem:3.4}
Let $q > p + 1$ and consider two appropriate sequences $(a_0, \ldots, a_q)$ and $(b_0,\ldots, b_p)$.
Let $u, v$ be non-negative integers satisfying $u < v$, $u + v \leq q - p$.
Then
\beq
\label{eq:3.7}
\sum_{0 \leq i \leq p} a_{i + u} b_i \leq \sum_{0 \leq i \leq p} a_{i + v} b_i.
\eeq
\end{lemma}

We are convinced that this is a known inequality but could not find a reference.
Therefore we provide a proof.

\begin{proof}[Proof of \eqref{eq:3.7}]
Let us first assume that both sequences are sequences of $0$ and $1$ only.
Appropriateness implies that the $1$'s constitute a symmetric interval in the middle.
The formula \eqref{eq:3.7} counts the number of overlaps between translates of $(a_0, \ldots, a_q)$ and $(b_0,\ldots, b_p)$.
Clearly the maximum is attained (not necessarily uniquely) when $u = \left\lfloor \frac{q - p}{2}\right\rfloor$ and $u = \left\lceil \frac{q - p}{2}\right\rceil$.
Now \eqref{eq:3.7} follows by ``discrete continuity''.

To prove \eqref{eq:3.7} just note that every appropriate sequence can be expressed as a non-negative linear combination of appropriate $0$-$1$-sequences.
\end{proof}

Let us fix $\ell \geq 2$ and consider a maximal chain $A_1 \subset A_2 \subset \ldots \subset A_\ell$ where $A_i \in {[\ell + 1]\choose i}$, $1 \leq i \leq \ell$.
Set also $B_i = [\ell + 1] \setminus A_i$.

For a subset $P \subset [\ell + 1]$ define $\alpha_p = \left(\dfrac{|\mathcal F(P, [\ell + 1])|}{{n - \ell - 1\choose k - |P|}}\right)$ and $\beta_p = \left(\dfrac{|\mathcal G(P, [\ell + 1])|}{{n - \ell - 1\choose \ell - |P|}}\right)$.

\setcounter{proposition}{4}
\begin{proposition}
\label{prop:3.5}
\begin{align}
\label{eq:3.8}
&\sum_{\text{\rm chains}} \sum_{1 \leq i \leq \ell} {\ell + 1\choose i} {n - \ell - 1\choose k - i} \alpha_{A_i} + {\ell + 1\choose \ell + 1 - i} {n - \ell - 1\choose i - 1}\beta_{B_i}\\
&\qquad = (\ell + 1)! (|\mathcal F| + |\mathcal G|) \nonumber
\end{align}
where the outer summation is over all $(\ell + 1)!$ full chains.
\end{proposition}

\begin{proof}
Since $A \subset {[\ell + 1]\choose i}$ occurs in $i!(\ell + 1 - i)!$ full chains, \eqref{eq:3.8} follows easily from the definition.
\end{proof}

Let us note that for the families giving equality in \eqref{eq:3.6}, $\alpha_A = 1$ for $|A| \geq 2$ and $\beta_B = 1$ for $|B| = \ell$ and $0$ otherwise.

Since $\alpha_A = \beta_B = 1$ for $\ell$-subsets and it is $0$ for singletons, to prove \eqref{eq:3.4} via \eqref{eq:3.8} it is sufficient to prove that subject to the conditions $0 \leq \alpha_2 \leq \ldots \leq \alpha_{\ell - 1} \leq 1$, $0 \leq \beta_2 \leq \ldots \leq \beta_{\ell - 1} \leq 1$ and $\alpha_i + \beta_j \leq 1$ for $i + j \leq \ell + 1$,
\beq
\label{eq:3.9}
\sum_{2\leq i \leq \ell - 1}\!\! {\ell \!+\! 1\choose i}\! \left(\!{n\! -\! \ell \!- \!1\choose k - i} \alpha_i \! +\! {n \! -\! \ell\! -\! 1\choose i - 1}\beta_{\ell + 1 - i}\!\right)\! \leq\! \sum_{2 \leq i \leq \ell - 1} \!\! {\ell \! +\! 1\choose i} {n \!- \!\ell\! - \!1\choose k - i}\!.
\eeq
I.e., the LHS of \eqref{eq:3.9} is maximized for $\alpha_2 = \ldots = \alpha_{\ell - 1} = 1$, $\beta_2 = \ldots = \beta_{\ell - 1} = 0$.
This is a LP problem for these variables.
Setting $\beta_{\ell + 1 - i} = 1 - \alpha_i$ can only increase the LHS.
If $\alpha_2 = 1$, we are done.
Otherwise let $t \leq \ell - 1$ be the maximal value of $i$ such that $\alpha_t < 1$, $\beta_t = 1 - \alpha_t > 0$.
Suppose that $t \geq \frac{\ell + 1}{2}$.
All we have to show is that changing $\alpha_i$ to $\alpha_i + \beta_t$ and $\beta_i$ to $\beta_i - \beta_t$ for $\ell + 1 - t \leq i \leq t$ does not decrease the LHS.
This is equivalent to
\beq
\label{eq:3.10}
\sum_{\ell + 1 - t \leq i \leq t} {\ell + 1\choose i} {n - \ell - 1\choose k - i} \geq \sum_{\ell + 1 - t \leq i \leq t} {\ell + 1\choose i} {n - \ell - 1\choose \ell - i}.
\eeq
Here both ${\ell + 1\choose 2}, \ldots, {\ell + 1\choose \ell - 1}$ and ${n - \ell - 1\choose 0},\ldots, {n - \ell - 1\choose n - \ell - 1}$ are appropriate sequences.
To have the sums in the same form as in \eqref{eq:3.7}, set $w = n - k - \ell - i$ and use
${\ell + 1\choose i} = {\ell + 1\choose \ell + 1 - i}$.
This turns \eqref{eq:3.10} into
\beq
\label{eq:3.11}
\sum_{\ell + 1 - t \leq i \leq t} {\ell + 1\choose \ell + 1 - i} {k + w\choose k - i} \geq \sum_{\ell + 1 - t \leq i \leq t} {\ell + 1\choose \ell + 1 - i} {k + w\choose \ell - i}
\eeq
Then the values of $k - i$ cover the interval $[k - t, k + t - \ell - 1]$ and the values of $\ell - i$ the interval $[\ell - t, t - 1]$.
Now $\ell - t + k - t \leq k - 1 < k + w$.
Hence \eqref{eq:3.11} follows from \eqref{eq:3.7}.

The case $t < \frac{\ell + 1}{2}$ is much easier.
Namely, for $n \geq k + \ell + 1$, $n - \ell - 1 \geq k$.
Thus for $2 \leq t < \frac{\ell + 1}{2}$, $t - 1 < \ell - t \leq k - t$ implying $1 \leq t - 1 < k - t$ and $(t - 1) + (k - t) = k - 1 < n - \ell - 1$.
Consequently, ${n - \ell - 1\choose t - 1} < {n - \ell - 1\choose k - t}$.
This shows that increasing $\alpha_t$ while diminishing $\beta_{\ell + i - k}$ by the same amount increases the LHS of \eqref{eq:3.9}.
Eventually we get $\alpha_2 = \ldots = \alpha_{\ell - 1} = 1$, $\beta_2 = \ldots = \beta_{\ell - 1} = 0$ proving \eqref{eq:3.9}.
Even more, equality holds only in this setting which proves \eqref{eq:3.6} together with uniqueness.
\hfill $\square$

\section{The proof of Theorem \ref{th:1.5}}
\label{sec:4}

Let us fix $k \geq \ell \geq 1$, $n \geq k + \ell$ and define $\mathcal G_0 = \{U, V\}$ where $U$ and $V$ are two disjoint $\ell$-subsets of $[n]$.
Set $\mathcal F_0 = \left\{F \in {[n]\choose k}: F \cap U \neq \emptyset, F \cap V \neq \emptyset \right\}$.

Obviously $\mathcal F_0$ and $\mathcal G_0$ are cross-intersecting and
\beq
\label{eq:4.1}
|\mathcal F_0| + |\mathcal G_0| = {n\choose k} - 2 {n - \ell\choose k} + {n - 2\ell \choose k} + 2 =: h(n, k, \ell).
\eeq
Moreover, for $\ell \geq 2$ both $\mathcal F_0$ and $\mathcal G_0$ are non-trivial.

Let $\mathcal F \subset {[n]\choose k}$, $\mathcal G \subset {[n]\choose \ell}$ be cross-intersecting.
In the case $n = k + \ell$, $h(n, k, \ell) = {k + \ell \choose k}$ and the bound $|\mathcal F| + |\mathcal G| \leq {n \choose k + \ell}$ is very easy to prove.
Just note that $F \in \mathcal F$ implies $[n] \setminus F \notin \mathcal G$.
If $\ell = 1$ and $|\mathcal G| = r > 1$ then $|\mathcal F| \leq {n - r\choose k - r}$ is obvious.
Now ${n - r\choose k - r} + r \leq {n - 2\choose k - 2} + 2 = h(n, k, 1)$.
I.e., in the case $\ell = 1$, \eqref{eq:4.1} provides the maximum of $|\mathcal F| + |\mathcal G|$ for cross-intersecting families subject to $|\mathcal G| \geq 2$ without requiring that $\mathcal F$ is non-trivial.
Our main result shows that \eqref{eq:4.1} provides the maximum in general.

\begin{theorem}
\label{th:4.1}
Let $k \geq \ell \geq 2$, $n > k + \ell$.
Suppose that $\mathcal F \subset {[n]\choose k}$ and $\mathcal G\subset {[n]\choose \ell}$ are non-trivial and cross-intersecting.
Then
\beq
\label{eq:4.2}
|\mathcal F| + |\mathcal G| \leq h(n, k, \ell).
\eeq
Moreover, unless $k = \ell = 2$, up to symmetry $\mathcal F_0, \mathcal G_0$ are the only families achieving equality in \eqref{eq:4.2}.
\end{theorem}

Let us note that in the case $k = \ell = 2$ there are two more essentially different constructions:
$\mathcal F_1 = \mathcal G_1 = {[3]\choose 2}$ and $\mathcal F_2 = \left\{(1,2) (2,3), (3,4)\right\}$,
$\mathcal G_2 = \left\{(1,3), (2,3), (2,4)\right\}$.
For fixed $k, \ell$ and $n \to \infty$, $h(n, k, \ell) = (\ell^2 - o(1)) {n - 2\choose k - 2}$ while $g(n, k, \ell)$ from Theorem \ref{th:3.3} satisfies $g(n, k, \ell) = \left({\ell + 1\choose 2} - o(1)\right) {n - 2\choose k - 2}$.

For large $n$ and $k \geq 3$ these show that $h(n, k, \ell) > g(n, k, \ell)$.
To prove it for all $n > k + \ell$, $k \geq 3$ is not evident.

\setcounter{proposition}{1}
\begin{proposition}
\label{prop:4.2}
Let $n > k + \ell$, $k \geq 3$, $\ell \geq 2$.
Then
\beq
\label{eq:4.3}
g(n, k, \ell) < h(n, k, \ell).
\eeq
\end{proposition}

Note that Theorem \ref{th:3.3} and \eqref{eq:4.3} imply Theorem \ref{th:4.1} for shifted (initial) families.
Nevertheless shifting (cf.\ Section \ref{sec:2}) plays a crucial role in the proof of Theorem \ref{th:4.1}.

Recall that $(i, j)$ and more generally $(x_1, \ldots, x_r)$ denote a set where we know that $i < j$, $x_1 < \ldots < x_r$.

Let us now prove \eqref{eq:4.3} in a more general form that will be useful later.

\begin{proposition}
\label{prop:4.3}
Suppose that $\mathcal F \subset {[n]\choose k}$, $\mathcal G \subset {[2\ell]\choose \ell}$ are non-trivial and cross-intersecting.
Let $n > k + \ell$, $k \geq \ell \geq 2$, $k \geq 3$.
Then
\beq
\label{eq:4.4}
|\mathcal F| + |\mathcal G| \leq h(n, k, \ell) \ \text{ and the inequality is strict unless } \ |\mathcal G| = 2.
\eeq
\end{proposition}

\begin{proof}
Let us define the family $\mathcal T_r$ of $r$-transversals of $\mathcal G$ by
$$
\mathcal T_r = \left\{ T \in {[2\ell]\choose r} : T \cap G \neq \emptyset \ \text{ for all } \ G \in \mathcal G\right\}, \quad 1 \leq r \leq k.
$$
Set $t_r = |\mathcal T_r|$.
Note that $t_r = {2\ell\choose r}$ for $\ell < r \leq k$.
The non-triviality of $\mathcal G$ implies $t_1 = 0$.
The maximality of $\mathcal F$ implies
\beq
\label{eq:4.5}
t_\ell = {2\ell\choose \ell} - |\mathcal G|.
\eeq

\setcounter{lemma}{3}
\begin{lemma}
\label{lem:4.4}
For $2 \leq r < \ell$,
\beq
\label{eq:4.6}
t_r \leq {2\ell \choose r} - 2 {\ell \choose r}.
\eeq
\end{lemma}

Let us prove \eqref{eq:4.4} assuming \eqref{eq:4.6}.
Note that the maximality of $|\mathcal F| + |\mathcal G|$ implies
$$
|\mathcal F| = \sum_{2\leq r \leq k} t_r {n - 2\ell\choose k - r}
$$
and using \eqref{eq:4.5}
$$
|\mathcal F| + |\mathcal G| = |\mathcal G| + \sum_{\substack{2\leq r < k\\ r\neq \ell}} t_r {n - k - \ell\choose k - r} + \left({2\ell \choose \ell} - |\mathcal G|\right) {n - 2\ell\choose k - \ell}.
$$
Setting $g = |\mathcal G| - 2$ and noting that for $\mathcal G_0 = \left\{[\ell], [\ell + 1, \ell]\right\}$ equality holds in \eqref{eq:4.6}
$$
|\mathcal F| + |\mathcal G| \leq h(n, k, \ell) - g \left({n - 2\ell\choose k - \ell} - 1\right) \leq h(n, k, \ell).
$$
\end{proof}

To prove \eqref{eq:4.6} we show
\beq
\label{eq:4.7}
{2\ell \choose r} - t_r \geq 2 {\ell \choose r},
\eeq
i.e., there are at least $2{\ell\choose r}$ subsets $R \in {[2\ell]\choose r}$ that are not transversals (covers) of~$\mathcal G$.
To this end pick an arbitrary $x \in [\ell]$ and a $G \in \mathcal G$ with $x \notin G$.
Then all ${\ell - 1\choose r - 1}$ subsets $R \in {[2\ell]\setminus G \choose r}$ satisfying $x \in R$ are non-covers.
Summing this over the $2\ell$ choices for $x$ we infer that there are at least $2\ell {\ell - 1\choose r - 1}\bigr/r = 2 {\ell \choose r}$ non-covers, as desired.
This concludes the proof of \eqref{eq:4.6}.

\smallskip
To prove \eqref{eq:4.1} we introduce a special way of simultaneously transforming the two families.
We call it \emph{shifting ad extremis}.
Let us say that two families $\mathcal F$ and $\mathcal G$ are \emph{CI} if they are \emph{cross-intersecting}.
We proved in \cite{F3} that this property is maintained by shifting, that is $S_{ij}(\mathcal F)$ and $S_{ij}(\mathcal G)$ are
CI as well.
Consequently we can keep on applying $S_{ij}$ simultaneously for arbitrary pairs $1 \leq i < j \leq n$.
However it might happen that $S_{ij}(\mathcal F)$ or $S_{ij}(\mathcal G)$ ceases to be non-trivial.

Throughout the whole proof we keep assuming that $|\mathcal F| + |\mathcal G|$ is maximal.
Consequently, whenever $\mathcal G(\overline i, \overline j) = \emptyset$ its counterpart $\mathcal F(i, j)$ must be \emph{full}, i.e., $\mathcal F(i, j) = \left\{F \in {[n]\choose k} : (i, j) \subset F\right\}$.
Similarly, if $\mathcal F(\overline i, \overline j) = \emptyset$ then $\mathcal G(i,j)$ is full.

The important thing about the shifting process is that if $S_{ij}(\mathcal F)$ or $S_{ij}(\mathcal G)$ is trivial (a star) then we renounce at $S_{ij}$ and choose an arbitrary different pair $(i', j')$ and perform $S_{i'j'}$.
We should stress three things.
First, by abuse of notation, we keep denoting the \emph{current} families by $\mathcal F$ and $\mathcal G$.
Second, since $i < j$, the sum $\sum\limits_{F \in \mathcal F} \sum\limits_{x \in F} x + \sum\limits_{G \in \mathcal G} \sum\limits_{y \in G} y$ keeps decreasing.
Third, we return to previously failed pairs $1 \leq i < j \leq n$, because it might happen that at a later stage in the process $S_{ij}$ does not destroy non-triviality any longer.
The important thing is that shifting ad extremis eventually produces two non-trivial families $\mathcal F$ and $\mathcal G$ that are CI and for each $1 \leq i < j \leq n$ one of (a), (b) and (c) holds.

\noindent
(a) \ $S_{ij}(\mathcal F) = \mathcal F$ \ and \ $S_{ij}(\mathcal G) = \mathcal G$.

\noindent
(b) \ $S_{ij}(\mathcal F)$ \ is a star.

\noindent
(c) \ $S_{ij}(\mathcal G)$ \ is a star.

Then we say that $(\mathcal F, \mathcal G)$ is \emph{shifted ad extremis}.
We are going to see below that this is a very strong property.

For $\mathcal H = \mathcal F$ or $\mathcal G$ let us define an ordinary graph $\mathcal T_2(\mathcal H)$ where $(i, j)$ is an edge of $\mathcal H$ if $\mathcal H(\overline i, \overline j) = \emptyset$.
By the maximality of $|\mathcal F| + |\mathcal G|$, for $\mathcal K = \{\mathcal F, \mathcal G\} \setminus \{\mathcal H\}$, $\mathcal K(i, j)$ is full.

\setcounter{fact}{4}
\begin{fact}
\label{fact:4.6}
$\mathcal T_2(\mathcal F)$ and $\mathcal T_2(\mathcal G)$ are cross-intersecting.
\end{fact}

\begin{proof}
Suppose the contrary and fix $(i, j) \in \mathcal T_2(\mathcal F)$ and $(a, b) \in \mathcal T_2(\mathcal G)$ that are disjoint.
By $n \geq k + \ell$ we can find $G \in {[n]\choose \ell}$ and $H \in {[n]\choose k}$ satisfying $G \cap \{i, j, a, b\} = (i, j)$, $H\cap \{i, j, a, b\} = (a, b)$ and $H \cap G = \emptyset$.

By the maximality of $|\mathcal G| + |\mathcal F|$, $G \in \mathcal G$ and $H \in \mathcal F$, a contradiction.
\end{proof}

We should note that if (b) holds for $(i, j)$ then $(i, j) \in \mathcal T_2(\mathcal F)$, but not necessarily vice versa.

If (a) holds for all $1 \leq i < j \leq n$ then $\mathcal F$ and $\mathcal G$ are initial and $|\mathcal F| + |\mathcal G| \leq g(n, k, \ell)$ follows from Theorem \ref{th:3.3}.
Consequently we may assume that (b) or (c) holds for at least one pair $(i, j)$.

Let us fix such an $(i, j)$ with $S_{ij}(\mathcal G)$ being a star.
The same argument would apply to the case (b) as well.

Since $\mathcal G(\overline i, \overline j) = \emptyset$ but $\mathcal G(\overline i)$ and $\mathcal G(\overline j)$ are non-empty,we can choose $K, H \in  \mathcal G$ with $K \cap (i, j) = \{i\}$, $H \cap \{i, j\} = \{j\}$.
Let us fix these $K$ and $H$ maximizing $|K \cap H|$.
In view of Fact \ref{fact:2.1}, $|K \cap H| \leq \ell - 2$.

\setcounter{lemma}{5}
\begin{lemma}
\label{lem:4.7}
Suppose that $x \neq i$, $y \neq j$ but $x \in K\setminus H$, $y \in H\setminus K$.
Then $\{x, y\}$ is in {\rm (c)}.
\end{lemma}

\begin{proof}
To simplify notation assume $x < y$ for the proof.
The case $y < x$ can be treated in the same way.
Since $S_{xy}(H) =: H' = (H \setminus \{y\}) \cup \{x\}$ and $|H' \cap K| = |H \cap K| \cup \{x\}$, $H' \notin \mathcal G$.
Consequently, $(x,y)$ cannot be of type (a).
In view of $(i, j ) \cap (x, y) = \emptyset$, it is not of type (b) either.
\end{proof}

At this point let us clarify our plan for proving \eqref{eq:4.1}.
There is a Plan A and a Plan~B.
Plan A is to find a pair $(i,j)$ of type (c) with the additional property that $\mathcal F(\overline i, \overline j)$ is non-trivial.
Then use induction.

Accordingly Plan B relates to the case that no such $(i,j)$ exists.
In this case we show that $\mathcal F$ and $\mathcal G$ are of a rather restricted type, e.g., they are shifted on $[n] \setminus \{i, j, x, y\}$, $\{i, j, x, y\}$ is a transversal for $\mathcal F$ etc.
Eventually we use these structural properties to prove \eqref{eq:4.1}.

\medskip
\noindent
\boxed{\text{\bf Case\ A.}}
There are $1 \leq i < j \leq n$ so that $(i, j)$ is of type (c) and $\mathcal F(\overline i, \overline j)$ is non-trivial.

\smallskip
We assume that \eqref{eq:4.1} has been proved for all pairs $k'\geq \ell' \geq 2$ with $k'+ \ell' < k + \ell$ and apply induction.
As noted before, for $\ell = 1$ and $k \geq 2$, \eqref{eq:4.1} holds whenever $\mathcal F, \mathcal G$ are CI and $|\mathcal G| \geq 2$.
This can serve as a base for the induction.

To prove \eqref{eq:4.1} we first note that
\beq
\label{eq:4.7m}
|\mathcal F(i, j)| + |\mathcal G(\overline i, \overline j)| = {n - 2\choose k - 2}.
\eeq
Applying Theorem \ref{th:1.7} to the non-empty CI families $\left(\mathcal F(i, \overline j), \mathcal G(\overline i, j)\right)$ and
$\left(\mathcal F(\overline i, j), \mathcal G(i, \overline j)\right)$ yields
\begin{align}
\label{eq:4.8}
\left|\mathcal F(i, \overline j)\right| + \left|\mathcal G(\overline i, j)\right| &\leq {n - 2\choose k - 1} - {n - \ell - 1\choose k - 1} + 1,\\
\label{eq:4.9}
\left|\mathcal F(\overline i, j)\right| + \left|\mathcal G(i, \overline j)\right| &\leq {n - 2\choose k - 1} - {n - \ell - 1\choose k - 1} + 1.
\end{align}
Moreover, the inequalities are strict unless $\left|\mathcal G(\overline i, j)\right| = 1$, $\left|\mathcal G(i, \overline j)\right| = 1$, respectively.
Finally, we need to prove
\beq
\label{eq:4.10}
\left|\mathcal F(\overline i, \overline j)\right| + \left|\mathcal G(i, j)\right| \leq {n - 2\choose k} - 2{n - \ell - 1\choose k} + {n - 2\ell\choose k}.
\eeq
To prove \eqref{eq:4.10} using the induction hypothesis we construct a non-trivial family $\widetilde{\mathcal G} \subset {[n]\setminus (i, j)\choose \ell - 1}$ satisfying $\bigl|\widetilde{\mathcal G}\bigr| \geq |\mathcal G(i, j)| + 2$.

First note that the non-triviality of $\mathcal G$ and $\mathcal G(\overline i, \overline j) = \emptyset$
and $(i, j)$ being of type (c)
imply that $\mathcal G(\overline i, j)$ and $\mathcal G(i, \overline j)$ are disjoint and $\mathcal G(\overline i, j) \cup \mathcal G(i, \overline j) \cup \mathcal G(i, j)$ is non-trivial on $[n] \setminus (i, j)$.
In the case $\mathcal G(i, j) = \emptyset$, we set $\widetilde{\mathcal G} = \mathcal G(\overline i, j) \cup \mathcal G(i, \overline j)$ and apply the induction hypothesis to the pair $\left(\mathcal F(\overline i, \overline j), \widetilde{\mathcal G}\right)$.
Noting $\bigl|\widetilde{\mathcal G}\bigr| \geq 2$, \eqref{eq:4.10} follows.
If $\mathcal G(i, j) \neq \emptyset$ let $\widehat{\mathcal G}$ be its \emph{shade},
$$
\widehat{\mathcal G} = \left\{\widehat G \in {[n]\setminus (i,j)\choose \ell - 1} : \exists G \in \mathcal G(i, j), G \subset \widehat G\right\}.
$$
By Sperner's Shadow Theorem,
\beq
\label{eq:4.11}
\bigl|\widehat{\mathcal G}\bigr| \geq \left|\mathcal G(i, j)\right| \frac{n - \ell}{\ell - 1}.
\eeq

We define $\widetilde{\mathcal G} = \widehat{\mathcal G} \cup \mathcal G(\overline i, j) \cup \mathcal G(i, \overline j)$.
Obviously $\widetilde{\mathcal G} \subset {[n]\setminus (i, j)\choose \ell - 1}$ and the pair $\left(\mathcal F(\overline i, \overline j), \widetilde{\mathcal G}\right)$ is CI.
The final piece to prove is
\beq
\label{eq:4.12}
\bigl|\widetilde{\mathcal G}\bigr| \geq |\mathcal G(i, j)| + 2.
\eeq
Since for $n > k + \ell$ every $(\ell - 2)$-subset of $[n] \setminus (i, j)$ is contained in $(n - 2) - (\ell - 2) > k \geq \ell$ subsets of size $\ell - 1$, $\bigl|\widehat{\mathcal G}\bigr| \geq \ell + 1$.
Thus in proving \eqref{eq:4.12} we may assume $|\mathcal G(i, j)| \geq \ell$.
From \eqref{eq:4.11} we infer
$$
\bigl|\widetilde{\mathcal G}\bigr| \geq |\mathcal G(i, j) | + |\mathcal G(i,j)| \frac{n - 2\ell + 1}{\ell - 1} \geq |\mathcal G(i, j)| + \frac{\ell}{\ell - 1} (n - 2\ell + 1) > |\mathcal G(i, j)| + 1
$$
proving \eqref{eq:4.12}. \hfill $\square$

\begin{rema}
The above proof shows that \eqref{eq:4.1} holds in the case $n = k + \ell + 1$ even if $\mathcal F(\overline i, \overline j)$ is trivial.
Indeed, the RHS of \eqref{eq:4.10} is ${k + \ell - 1\choose k} - 2$.
Above we constructed $\widetilde{\mathcal G} \in {[n]\choose \ell - 1}$ so that $\mathcal F(\overline i, \overline j)$ and $\widetilde{\mathcal G}$ are CI on the $(k + \ell - 1)$-set $[n]\setminus (i, j)$.
Thus $|\mathcal F(\overline i)| + \bigl|\widetilde{\mathcal G}\bigr| \leq {k + \ell - 1\choose k}$.
Hence \eqref{eq:4.12} implies \eqref{eq:4.10}.
\end{rema}

From now on we assume $n \geq k + \ell + 2$.

\medskip
\noindent
\boxed{\text{\bf Case\ B.}} For all pairs $(i, j)$ of type (c), $\mathcal F(\overline i, \overline j)$ is a star.

Let $z(i, j)$ denote an element common to all members of $\mathcal F(\overline i, \overline j)$.
Let us set $Z = \{i, j, x, y\}$ from Lemma \ref{lem:4.7}.
We claim that $(a,b)$ is of type (a) for all $(a, b) \subset [n]\setminus Z$.

First note that $(a, b)$ is not of type (b) because of Fact \ref{fact:4.6}.
Should it be of type (c), we infer $(a, b) \in \mathcal T_2(\mathcal G)$ whence all $k$-sets $F \subset [n]$ satisfying $F \cap \{i, j, a, b\} = (a, b)$ are in $\mathcal F(\overline i, \overline j)$.
Since the same holds for all $k$-sets with $F \cap \{i, j, x, y\} = \{x, y\}$, $\mathcal F(\overline i, \overline j)$ is non-trivial, a contradiction.

Note that via Lemma \ref{lem:4.7}, $|K \cap H| = \ell - 2$ follows as well.
Set $W = [n]\setminus Z = (w_1, \ldots, w_{n - 4})$.
Define $\widetilde K = \left\{i, x, w_1, \ldots, w_{\ell - 2}\right\}$,
$\widetilde H = \left\{j, y, w_1, \ldots, w_{\ell - 2}\right\}$.
Since all pairs in $W$ are of type (a), $K, H \in \mathcal G$ imply $\widetilde K, \widetilde H \in \mathcal G$.
WLOG assume $K = \widetilde K$, $H = \widetilde H$.

\setcounter{fact}{6}
\begin{fact}
\label{fact:4.8}
If $|L \cap Z| \leq 2$ for some $L \in \mathcal G$ then $L \cap Z = \{i, x\}$ or $\{j, y\}$.
\end{fact}

\begin{proof}
From $L \cap (i, j) \neq \emptyset$ and $L \cap (x,y) \neq \emptyset$, $|L \cap Z| \geq 2$ follows.
We have to show that $L \cap Z$ is neither $\{i, y\}$ nor $\{j, x\}$.
By symmetry assume $L \cap Z = \{i, y\}$.
By shiftedness on $W$, $\widetilde L = \left\{ i, y, w_1, \ldots, w_{\ell - 2}\right\} \in \mathcal G$.
However this contradicts $\mathcal G(\overline x) \cap \mathcal G(\overline y) = \emptyset$.
\end{proof}

\setcounter{corollary}{7}
\begin{corollary}
\label{cor:4.9}
$\mathcal G(\overline x, \overline j) = \emptyset = \mathcal G(\overline i, \overline y)$.
\end{corollary}

\begin{proof}
If $G \in \mathcal G(\overline x, \overline j)$ then $|G \cap Z| \geq 2$ implies $G \cap Z = \{i, y\}$, a contradiction.
\end{proof}

\begin{corollary}
\label{cor:4.10}
$\{x, j\}$ and $\{i, y\}$ are of type {\rm (c)}.
\end{corollary}

\begin{proof}
By Corollary \ref{cor:4.9} and the maximality of $|\mathcal F| + |\mathcal G|$, both $\mathcal F(x, j)$ and $\mathcal F(i, y)$ are full.
This implies that $\{x, j\}$ and $\{i, y\}$ could not be of type (b).
Should, say, $\{x, j\}$ be of type (a), we infer according to whether $x < j$ or $j < x$ that $\left\{x, y, w_1, \ldots, w_{\ell - 2}\right\} = \mathcal G$ or $\left\{i, j, w_1, \ldots, w_{\ell - 2}\right\} \in \mathcal G$.
Both contradict Fact \ref{fact:4.8}.
\end{proof}

\setcounter{fact}{9}
\begin{fact}
\label{fact:4.11}
$Z$ is a transversal of $\mathcal F$.
\end{fact}

\begin{proof}
Since $\mathcal F(x,y)$ is full $z(i, j) \in \{x, y\}$.

However, if $F \in \mathcal F$ satisfies $F \cap Z = \emptyset$ then $F \in \mathcal F(\overline i, \overline j)$ but $z(i, j) \notin F$, a contradiction.
\end{proof}

Now we are in a position to prove an important lemma.

\setcounter{lemma}{10}
\begin{lemma}
\label{lem:4.12}
Set $V = Z \cup \left(w_1, \ldots, w_{k + \ell - 4}\right)$, $|V| = k + \ell$.
Then for all $F \in \mathcal F$ and $G \in \mathcal G$,
$$
F \cap G \cap V \neq \emptyset.
$$
\end{lemma}

\begin{proof}
Arguing for contradiction choose $F \in \mathcal F$, $G \in \mathcal F$ with $F \cap G \cap V = \emptyset$ and among such sets
$|F \cap G|$ is as small as possible.
Choose $v \in F \cap G$.
Then $F \cap Z \neq \emptyset$ implies $\left|F \cap \left(w_1, \ldots, w_{k + \ell - 4}\right)\right| \leq k - 2$ and by $|G \cap Z| \geq 2$, $\left|G \cap \left(w_1, \ldots, w_{k + \ell - 4}\right)\right| \leq \ell - 3$.

Consequently we can pick $w \in \left(w_1, \ldots, w_{k + \ell - 4}\right)$ with $w \notin F \cup G$.
By shiftedness $F' := \left(F \setminus \{z\}\right) \cup \{w\}$ is in $\mathcal F$.
However, $F' \cap G \cap V = \emptyset$ and $|F' \cap G| < |F \cap G|$, a contradiction.
\end{proof}

Seeing $\ell - 3$ in the proof the careful reader might wonder, what about the case $\ell = 2$~?
Actually, that case is readily settled by Fact \ref{fact:4.8}.
If $\ell = 2$ it implies $\mathcal G = \left\{\{i, x\}, \{j, y\}\right\}$ and \eqref{eq:4.1} follows together with uniqueness.

Let us use Lemma \ref{lem:4.12} to prove \eqref{eq:4.1} in an important special case.

\setcounter{proposition}{11}
\begin{proposition}
\label{prop:4.13}
Suppose that $\mathcal G(n) = \emptyset$ and $k > \ell$.
Then \eqref{eq:4.1} holds.
Moreover, inequality is strict unless $|\mathcal G| = 2$.
\end{proposition}

\begin{proof}
Note that the fullness of $\mathcal F(i,j)$ and $\mathcal F(x,y)$ imply that $\mathcal F(n)$ is non-trivial.
In view of Lemma \ref{lem:4.12}, $\mathcal F(n)$ and $\mathcal G$ are CI.
The same is true for $\mathcal F(\overline n)$ and $\mathcal G$.
Thus
\begin{align*}
|\mathcal F(\overline n)| + |\mathcal G| - 2 &\leq {n - 1\choose k} - 2 {n - \ell - 1\choose k} + {n - 2\ell - 1\choose k},\\
|\mathcal F(n)| + |\mathcal G| - 2 &\leq {n - 1\choose k} - 2 {n - \ell - 1\choose k - 1} + {n - 2\ell - 1\choose k - 1}.
\end{align*}

Adding these yields
$|\mathcal F| + |\mathcal G| \leq {n \choose k} - 2 {n - \ell\choose k} + {n - 2\ell\choose k} + 2 - (|\mathcal G| - 2)$.
\end{proof}

As useful as Proposition \ref{prop:4.13} might look, it is not needed in the rest of the proof.
We believe that it might be useful in other situations.
On the other hand one can use Lemma \ref{lem:4.12} to settle the case $k = \ell$ as well.

\setcounter{claim}{12}

\begin{claim}
\label{cl:4.14}
If $\mathcal G(n) = \emptyset$ then $G \subset V$ for all $G \in \mathcal G$.
\end{claim}

\begin{proof}
Suppose that $u \in G \setminus V$ for some $G \in \mathcal G$.
Then Lemma \ref{lem:4.12} implies $G \setminus \{u\} \cap F \neq \emptyset$ for all $F \in \mathcal F$.
Consequently the same holds for $(G \setminus \{u\}) \cup \{n\} =: G'$.
Now the maximality of $|\mathcal G| + |\mathcal F|$ implies $G' \in \mathcal G$.
Hence $\mathcal G(n) \neq \emptyset$.
\end{proof}

If $k = \ell$ and $\mathcal G(n) = 0$ then Claim \ref{cl:4.14} and Lemma \ref{lem:4.4} imply \eqref{eq:4.1} together with uniqueness.

Recall that for $a \in \{i, x\}$, $b \in \{j, y\}$, $(a, b)$ is of type (c) and $\mathcal F(\overline a, \overline b)$ is trivial.
Let $z = z(a, b)\in F$ for all $F \in \mathcal F(\overline a, \overline b)$.
This implies that $\{a, b, z\} \cap F \neq \emptyset$ for all $F \in \mathcal F$.
That is, for each of the four choices of $(a, b)$ there is a $T$, $(a, b) \subset T \in {Z \choose 3}$ where $T$ is a transversal of~$\mathcal F$.
Using this for $(a, b) = (i, j)$ and $\{x, y\}$ we infer that at least two of the four sets in ${Z\choose 3}$ are transversals of $\mathcal F$.
Could it be three or four?

\setcounter{proposition}{13}
\begin{proposition}
\label{prop:4.15}
Exactly two members, $T, T' \in {Z\choose 3}$ are transversals of $\mathcal F$.
Moreover, either $\{T, T'\} = \bigl\{Z \setminus \{x\}, Z\setminus \{i\}\bigr\}$ or $\{T, T'\} = \bigl\{Z \setminus \{y\}, Z \setminus \{j\}\bigr\}$.
\end{proposition}

\begin{proof}
If $T \in {Z \choose 3}$ is a transversal of $\mathcal F$ then saturatedness implies that all $\ell$-sets containing $T$ are in $\mathcal G$.
In particular, $G_T := T \cup \{w_1, \ldots, w_{\ell - 3}\}$.
For a transversal $T \in {Z \choose 3}$ define $z(T)$ by $T = Z\setminus \{z(T)\}$.
If for two transversals $T, T'$, $\{z(T), z(T')\} = \{i, j\}$ or $\{x, y\}$ then $G(T)$ and $G(T')$ contradict $\mathcal G(i) \cap \mathcal G(j) \neq \emptyset$ or $\mathcal G(x) \cap \mathcal G(y) \neq \emptyset$, respectively.
This implies that we could not have three or more such $T$.
Finally, if $\{z(T), z(T')\} = \{i, y\}$ or $\{j, x\}$ then there is no transversal of size $3$ containing $\{i, y\}$ or $\{j, x\}$, respectively.
This concludes the proof.
\end{proof}

Let us suppose that $T = \{i, x, j\}$ and $T' = \{i, x, y\}$.

\setcounter{claim}{14}
\begin{claim}
\label{cl:4.16}
There exist $F, F' \in \mathcal F$ such that $F \cap Z = \{i\}$ and $F' \cap Z = \{x\}$.
\end{claim}

\begin{proof}
By Proposition \ref{prop:4.15}, there is some $F \in \mathcal F$ satisfying $F \cap \{x, j, y\} = \emptyset$.
By Fact \ref{fact:4.11}, $F \cap Z = \{i\}$.
Considering the non-transversal $\{i, j, y\}$ we obtain $F' \in \mathcal F$ with $F' \cap Z = \{x\}$.
\end{proof}

Since $T \cap H = \{j\}$, $\mathcal G(j)$ is non-trivial.
Let us show:

\setcounter{claim}{15}
\begin{claim}
\label{cl:4.17}
$\mathcal F(\overline j)$ is non-trivial.
\end{claim}

\begin{proof}
Since $\mathcal F(i, x)$ is full, the only candidates for membership in $\cap \mathcal F(\overline j)$ are $i$ and $x$.
However, $F, F'$ from Claim \ref{cl:4.16} are in $\mathcal F(\overline j)$ and $F \cap F' \cap Z = \emptyset$.
\end{proof}

Applying the induction hypothesis to $\mathcal F(\overline j)$ and $\mathcal G(j)$ yields
\beq
\label{eq:4.13}
|\mathcal F(\overline j)| + |\mathcal G(j)| \leq {n - 1\choose k} - 2 {n - \ell \choose k} + {n - 2\ell + 1\choose k} + 2.
\eeq
Here we distinguish two cases, namely $k > \ell$ and $k = \ell$.
If $k > \ell$, i.e., $k - 1 \geq \ell$ then we apply \eqref{eq:1.7} to the non-empty CI pair $\mathcal F(j)$ and $\mathcal G(\overline j)$:
\beq
\label{eq:4.14}
|\mathcal F(j)| + |\mathcal G(\overline j)| \leq {n - 1\choose k - 1} - {n - \ell - 1\choose k - 1} + 1.
\eeq

The sum of \eqref{eq:4.13} and \eqref{eq:4.14} is:
$$
|\mathcal F| + |\mathcal G| \leq {n\choose k} - 2{n - \ell \choose k} + {n - 2\ell\choose k} + 2 - \left(\!{n \!-\! \ell \! -\! 1\choose k - 1} - {n - 2\ell\choose k - 1} \! -\! 1\!\right).
$$
That is, to prove \eqref{eq:4.1} we need to show,
$$
{n - \ell - 1\choose k - 1} - {n - 2\ell \choose k - 1} = {n - \ell - 2\choose k - 2} + {n - \ell - 3\choose k - 2} + \ldots {n - 2\ell \choose k - 2} > 1.
$$
This is true for $n \geq k + \ell + 1$, $k \geq 3$ by ${n - \ell - 2\choose k - 2} \geq {k - 1\choose k - 2} = k - 1 > 1$.

The second case is $k = \ell$.

Since $\{i, x, y\}$ is a transversal of $\mathcal F$,
$\mathcal F(j)| = |\mathcal F_1| + |\mathcal F_2|$ where
$$
\mathcal F_1 = \bigl\{F \in \mathcal F, j \in F, F \cap \{i, x\} \neq \emptyset\bigr\} \ \ \text{ and } \ \ \mathcal F_2 = \mathcal F(\{y, j\}, Z).
$$
For $G \in \mathcal G(\overline j)$, Fact \ref{fact:4.8} implies $\{i, x\} \subset G$.
Thus
$$
|\mathcal G(\overline j)| = \bigl|\mathcal G(\{i, x, y\}, Z)\bigr| + \bigl|\mathcal G(\{x, y\}, Z)\bigr|.
$$
Obviously,
$$
\gathered
|\mathcal F_1| \leq {n - 1\choose k - 1} - {n - 3\choose k - 1},\\
\bigl|\mathcal G(\{i, x, y\}, Z)\bigr| \leq {n - 4\choose k - 3}.
\endgathered
$$
The important observation is that $\mathcal F(\{y, j\}, Z)$ and $\mathcal G(\{i, x\}, Z)$ are cross-intersecting $(k - 2)$-graphs on $[n] \setminus Z$.
By \eqref{eq:2.2}, $\bigl|\mathcal F(\{y, j\}, Z)\bigr| + \bigl|\mathcal G(\{i, x\}, Z)\bigr| \leq {n - 4\choose k - 2}$.
Thus we infer
\beq
\label{eq:4.15}
|\mathcal F(j)| + |\mathcal G(\overline j)| \leq {n - 1\choose k - 1} - {n - 3\choose k - 1} + {n - 3\choose k - 2}.
\eeq
Adding this to \eqref{eq:4.13} yields
$$
|\mathcal F| + |\mathcal G| \leq \! {n\choose k} - 2{n\! -\! k\choose k} + {n \!-\! 2k\choose k} + 2 - {n\! -\! 2k \choose k - 1} - \left(\!{n \! -\! 3\choose k\! -\! 1} - {n\! -\! 3\choose k\! -\! 2}\!\right),
$$
proving \eqref{eq:4.1} with strict inequality for $n > 2k$.
\hfill $\square$

\section{The new proof of Theorem \ref{th:1.7}}
\label{sec:5}

Since simultaneous shifting maintains cross-intersection we may assume that $\mathcal F$ and $\mathcal G$ are initial families.
For $G \in \mathcal G$ define the quantity $p(G)$ as the maximal integer $p$ with the property
\beq
\label{eq:5.1}
\bigl|G \cap [2p + k - \ell]\bigr| \geq p.
\eeq

Note that \eqref{eq:5.1} is always satisfied for $p = 0$.
This implies $0 \leq p \leq \ell$.

In the case $k = \ell$, should $p(G) = 0$ hold for some $G$, then $(3, 5, 7, \ldots, 2k + 1) \prec G$ follows.
Indeed, otherwise if $G = (x_1, \ldots, x_k)$ then for some $1 \leq p \leq k$, $x_p \leq 2_p$ and thereby $|G \cap [2p]| \geq p$ would hold.

Thus $p(G) = 0$ and shiftedness imply $(3, 5, \ldots, 2k + 1) \in \mathcal G$.
We claim that $p(F) > 0$ for all $F \in \mathcal F$.
In the opposite case we infer that $(3, 5, \ldots, 2k + 1) \in \mathcal F$.
However, $\mathcal F$ is initial, yielding $(2, 4, \ldots, 2k) \in \mathcal F$ which contradicts cross-intersection.

In the $k = \ell$ case, if necessary we interchange $\mathcal F$ and $\mathcal G$.
Then we can suppose that $p(G) > 0$ for all $G \in \mathcal G$.

Let us now define the map $\varphi : \mathcal G \to {[n]\choose k}$ by
$\varphi(G) = G \Delta [2p(G) + k - \ell]$ \ ($\Delta$ denotes symmetric difference).

\begin{lemma}
\label{lem:5.1}
{\rm (i)} \ \ \ $|\varphi(G)| = k$,\\
\hspace*{40mm} {\rm (ii)} \ \ $\varphi$ is an injection with $\varphi(G) \notin \mathcal F$,\\
\hspace*{40mm} {\rm (iii)} \  $\varphi(G) \cap [\ell] \neq \emptyset$ for $G \neq [\ell]$.
\end{lemma}

\begin{proof}
The maximal choice of $p$ in \eqref{eq:5.1} implies $\bigl|G \cap [2p(G) + k - \ell]\bigr| = p(G)$.
Thus $\bigl|G \Delta [2p(G) + k - \ell]\bigr| = |G| + k - \ell = k$, proving (i).

Let us show $\varphi(G) \neq \varphi(G')$ for $G \neq G' \in \mathcal G$.
If $p(G) = p(G')$ then this is evident.
Suppose that $p(G) > p(G')$.
Then $\varphi(G) \cap [2p(G) + k - \ell] = p(G) + k - \ell$.
However $\bigl|G' \cap 2p(G) + (k - \ell)\bigr| < p(G)$ implies $\varphi(G') \cap [2p(G) + k - \ell] =
p(G') + k - \ell + \bigl|G' \cap [2p(G') + k - \ell + 1, 2p(G) + k - \ell]\bigr| < p(G) + k - \ell$.
This shows that $\varphi$ is injective.

To prove $\varphi(G) \notin \mathcal F$ first note that the maximal choice of $p = p(G)$ implies that for $G = (x_1, \ldots, x_\ell)$,
$x_{p + i} > 2p + k - \ell + 2i$, $p < i \leq \ell$.
Using shiftedness for $\mathcal G$ we infer $\left(G \cap [2p + k - \ell] \cup \{k - \ell + 2(p + 1), k - \ell + 2(p + 2), \ldots, k - \ell + 2\ell\}\right) \in \mathcal G$.
If $\varphi(G) \in \mathcal F$ then the shiftedness of $\mathcal F$ implies in the same way
$$
\bigl([2p + k - \ell] \setminus G \cup \{k - \ell + 2p + 1, k - \ell + 2p + 3, \ldots, k - \ell + 2\ell - 1\}\bigr) \in \mathcal F.
$$
However these two sets are disjoint, a contradiction.

To prove (iii) is easy.
If $G \neq [\ell]$ then let $x$ be the minimal element of $[\ell] \setminus G$.
Now $[x - 1] \subset G$ implies $p(G) \geq x - 1$ and $2p(G) + k - \ell \geq x$ because $p(G) + k - \ell \geq 1$ either by $k > \ell$ or by $k = \ell$ and $p(G) \geq 1$.
Hence $x \in \varphi(G)$.
\end{proof}

Let us deduce Theorem \ref{th:1.7} from the lemma.
By shiftedness $[\ell] \in \mathcal G$.
Set $\mathcal H = \left\{H \in {[n]\choose k} : H \cap [\ell] \neq \emptyset \right\}$.
By cross-intersection $\mathcal F \subset \mathcal H$.
By the lemma $\varphi(\mathcal G \setminus \{[\ell]\}) \cap \mathcal F = \emptyset$ and $\varphi(\mathcal G) \subset \mathcal H$ as well.
Consequently
\beq
\label{eq:5.2}
|\mathcal F| + |\mathcal G| - 1 \leq |\mathcal H| = {n \choose k} - {n - \ell\choose k} \ \text{ proving \eqref{eq:1.7}}.
\eeq

Let us show that the inequality is strict if $|\mathcal G| > 1$ unless $k = \ell = 2$.
By shiftedness $[\ell + 1] \setminus \{\ell\} =: G_\ell \in \mathcal G$.
Obviously, $p(G_\ell) = \ell$.
Thus $\varphi(G_\ell) = (\ell, \ell + 2, \ell + 3, \ldots, k + \ell) =: H_0 \notin \mathcal F$ by Lemma \ref{lem:5.1} (ii).

\begin{proof}
Define $H_1 = (\ell, \ell + 2) \cup [\ell + 4, k + \ell + 1]$.
Note that $k \geq 3$ implies $[\ell + 4, k + \ell + 1] \neq \emptyset$.
Since $H_0 \prec H_1$, $H_1 \notin \mathcal F$.
Should equality hold in \eqref{eq:1.7}, that is, in \eqref{eq:5.2}, there is some $G_1 \in \mathcal G$ with $\varphi(G_1) = H_1$.

Now $H_1 = G_1 \Delta\bigl[2\ell(G_1) + k - \ell\bigr]$ implies $H_1 \Delta \bigl[2\ell(G_1) + k - \ell\bigr] = G_1$.

Using $\ell(G_1) + k - \ell > 0$, $1 \in G_1$ follows.
This implies $\ell(G_1) \geq 1$ and $2\ell(G_1) + k - \ell \geq 2$.
Using also $[\ell - 1] \cap H_1 = \emptyset$ we can prove successively $\ell(\mathcal G_1) > \ell - 1$ and therefore $\ell(\mathcal G_1) = \ell$.
However $H_1 \Delta [k + \ell] = [\ell - 1] \cup \{\ell + 1, \ell + 3\}$ is an $(\ell + 1)$-set contradicting $H_1 = \varphi(G_1)$.
This concludes the proof.
\end{proof}

We should mention that in the case $k = \ell = 2$, $H_1 = (2,4)$ and $k + \ell = 4$.
Thus $H_1 = \varphi(G_1)$ with $G_1 = (1,3)$. As a matter of fact, setting
$\mathcal F = \mathcal G = \{(1, i) : 2 \leq i \leq n\}$ gives equality in \eqref{eq:1.7}.

\section{Concluding remarks}
\label{sec:6}

In the present paper we considered the problem of determining the maximum of $|\mathcal F| + |\mathcal G|$ for families $\mathcal F, \mathcal G \subset {[n]\choose k}$ where $\mathcal F$ and $\mathcal G$ are cross-intersecting and non-trivial.

Recently, in a joint paper with Jian Wang \cite{FW} we proved the following result concerning the product $|\mathcal F||\mathcal G|$.

\begin{theorem}[\cite{FW}]
\label{th:6.1}
Let $k \geq 6$, $n > 9k$ and suppose that $\mathcal F, \mathcal G \subset {[n]\choose k}$ are cross-intersecting and non-trivial.
Then
\beq
\label{eq:6.1}
|\mathcal F||\mathcal G| \leq \left({n - 1\choose k - 1} - {n - k - 1\choose k - 1} + 1\right)^2.
\eeq
\end{theorem}

Note that in the above range \eqref{eq:6.1} implies \eqref{eq:1.2}, that is, the Hilton--Milner Theorem.

Note that \eqref{eq:6.1} can be proved easily for $k = 2$, $n \geq 4$.
However, we do not know whether it holds for all $(n, k)$ satisfying $n \geq 2k \geq 4$.

\end{document}